\documentclass[a4paper,12pt,reqno,oneside]{amsart} 

\usepackage{setspace} 
\usepackage{typearea} 
\usepackage{amscd,amsmath,amssymb,verbatim} 

\usepackage[dvipsnames]{xcolor}
\usepackage{xr-hyper}
\usepackage{cite}
\usepackage[colorlinks,backref,final,bookmarksnumbered,bookmarks]{hyperref}
\usepackage[backrefs]{amsrefs}
\hypersetup{linkcolor=bluex,citecolor=Green,filecolor=cyan,urlcolor=magenta}
\usepackage{graphicx}
\definecolor{bluex}{Hsb}{220, 1, 1}

\newcommand{\borderref}[2]{{\hypersetup{linkcolor=black}\hyperref[#1]{#1 #2}}}

\usepackage{xfrac}

\usepackage{ifthen}
\newcommand{\arxiv}[2][]{\ifthenelse{\equal{#1}{}}
{\href{http://arxiv.org/abs/#2}{\tt arXiv:#2}}
{\href{http://arxiv.org/abs/math/#2}{\tt arXiv:math.#1/#2}}}

\theoremstyle{plain}
\newtheorem{theorem}{Theorem}[section]
\newtheorem*{theorem*}{Theorem}

\newtheorem{lemma}[theorem]{Lemma}
\newtheorem{proposition}[theorem]{Proposition}
\newtheorem{corollary}[theorem]{Corollary}
\newtheorem*{corollary*}{Corollary}

\newtheorem*{problem*}{Problem}
\newtheorem*{addendum*}{Addendum}

\theoremstyle{definition}
\newtheorem{remark}[theorem]{Remark}
\newtheorem*{remark*}{Remark}
\newtheorem{example}[theorem]{Example}
\newtheorem*{examples*}{Examples}

\newtheoremstyle{named}{}{}{\itshape}{}{\bfseries}{.}{.5em}{\thmnote{#3}}
\theoremstyle{named}

\def\x{\times}
\def\but{\setminus}

\def\phi{\varphi}
\def\emptyset{\varnothing}
\renewcommand{\:}{\colon}

\def\Z{\mathbb{Z}}

\def\xr#1{\xrightarrow{#1}}

\makeatletter
\newcommand{\xR}[2][]{\ext@arrow 0359\Rightarrowfill@{#1}{#2}}
\newcommand{\xL}[2][]{\ext@arrow 0359\Leftarrowfill@{#1}{#2}}
\makeatother

\DeclareMathOperator{\lk}{lk}

\DeclareMathOperator{\id}{id}

\begin{document}
\title{Local knots and the prime factorization of links} 
\author{Sergey A. Melikhov}
\address{Steklov Mathematical Institute of Russian Academy of Sciences, 8 Gubkina st., 119991 Moscow, Russia}
\email{melikhov@mi-ras.ru}

\begin{abstract}
The present note contains a new proof of Y. Hashizume's 1958 theorem that every non-split link in $S^3$ 
admits a unique factorization into prime links. 
While the new proof does not go far beyond standard techniques, it is considerably shorter than 
the original proof and avoids most of its case exhaustion.

We apply this proof to obtain a string link version (and also an alternative proof) of a 1972 theorem of D. Rolfsen:
two PL links in $S^3$ are ambient isotopic if and only if they are PL isotopic and their respective 
components are ambient isotopic.
It is tempting to dismiss this string link version as obvious by deriving it directly either from
Rolfsen's or Hashizume's theorem.
But this does not seem to be possible, as it turns out that there exists a string link that has 
no local knots, while its closure has a local knot.
\end{abstract}

\maketitle

\section{Introduction} \label{intro}

By a {\it link} we will mean, unless noted otherwise, a PL link in $S^3$, that is, a piecewise linear embedding $mS^1\to S^3$, 
where $mS^1=\{1,\dots,m\}\x S^1$.

Given an $m$-component link $L$, an $n$-component link $L'$ and integers $i\in\{1,\dots,m\}$ and $j\in\{1,\dots,n\}$,
there is an $(m+n-1)$-component link $L\#_{i,j}L'$, which is well-defined up to an ambient isotopy (see \S\ref{factorization} 
for the details).
This ``connected sum along selected components'' was introduced by Y. Hashizume \cite{Has}, who also obtained the
following result (see Theorem \ref{hashizume} for a more detailed statement).

\begin{theorem}[Hashizume \cite{Has}] \label{hashizume0}
Every non-split link in $S^3$ admits a unique decomposition into prime factors (well-defined up to ambient isotopy
and up to a permutation of the factors) with respect to the operation of Hashizume connected sum.
\end{theorem}

Theorem \ref{hashizume0} plays a fundamental role in knot theory as it is needed to make sense out of the knot and link tables, which 
traditionally list only prime links.
Nevertheless I was unable to find its proof in any textbook, even though a number of textbooks include the proof of its
best known special case --- the prime factorization of knots.
One potential reason is that textbook authors might be somewhat terrified by Hashizume's original proof \cite{Has}, 
which takes about 16 pages.
(About a half of these 16 pages consist of pictures, which could be related to the fact that much of the proof proceeds by
exhaustion of cases.)
In \S\ref{factorization} we provide an alternative proof of Theorem \ref{hashizume0}, which aims at being a bit more conceptual.
The main idea, which comes from \cite{MR2}*{proof of Lemma 2.2}, is to work with multiply punctured $3$-spheres.

By a {\it tangle} we will mean a proper PL embedding $L\:\Theta\to M$ of a compact $1$-manifold in a $3$-manifold, 
where {\it proper} means that $L^{-1}(\partial M)=\partial\Theta$.
Thus tangles include links and string links.
A tangle $L\:\Theta\to M$ is said to {\it have a local knot} if there exists a PL $3$-ball $B$ in $M$ that meets $L(\Theta)$ 
in an arc such that the union of this arc with an arc in $\partial B$ that has the same endpoints is a nontrivial knot.
A tangle is called {\it totally split} if its components are contained in pairwise disjoint $3$-balls.

Hashizume's connected sum should not be confused with ``componentwise connected sum'', which is not well-defined as such, 
but has noteworthy variations that are well-defined.
Namely, if $K$ and $L$ are $m$-component links, then their componentwise connected sum
$L\#K$ is an $m$-component link obtained by placing $K$ and $L$ in disjoint $3$-balls and joining their respective
components along $m$ pairwise disjoint bands. 
If $K$ is totally split, then $K\# L$ is easily seen to be well-defined up to ambient isotopy.
Also, for any two $m$-component string links $L$ and $L'$, their componentwise connected sum $L\# L'$ is a well-defined
$m$-component string link.
A prime factorization theorem for componentwise connected sum of $2$-component string links is proved in \cite{BBK}.

\begin{theorem} \label{global}
Every (string) link admits a unique decomposition into a componentwise connected sum of a (string) link that has no local knots
and a totally split (string) link.
\end{theorem}

Here the case of string links is proved by using the new proof of Theorem \ref{hashizume0}, while the case of links is 
an easy consequence of Theorem \ref{hashizume0} itself (see \S\ref{local knots} for the details).
The following example suggests that the case of string links is not a direct consequence of the case of links.

\begin{figure}[h]
\includegraphics[width=\linewidth]{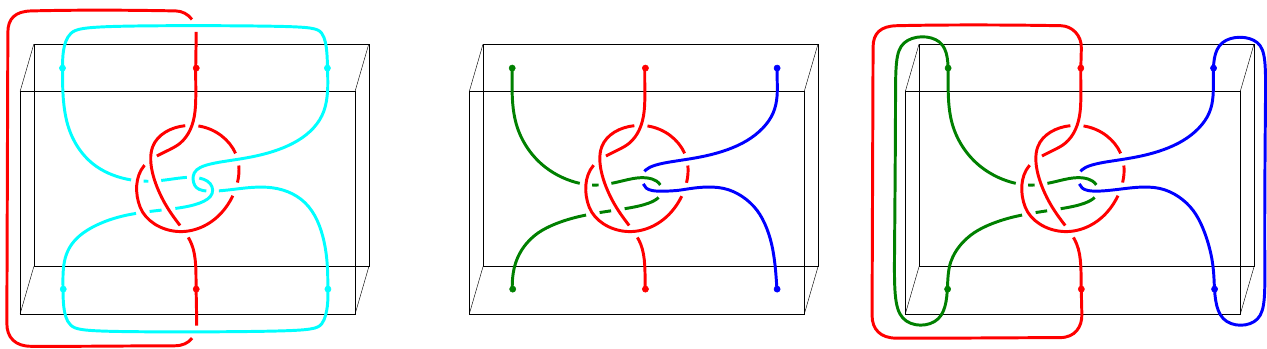}
\caption{A string link with no local knots whose closure has a local knot.}
\label{nonlocal}
\end{figure}

\begin{example} \label{nonlocal-knot}
There exists a $3$-component string link $L$ such that $L$ has no local knots, but its closure has a local knot.
See Figure \ref{nonlocal}, where $L$ is shown in the middle and its closure on the right.

Indeed, it is easy to see that the closure of $L$ has a local knot. 
To show that $L$ itself has no local knots, we consider a non-standard two-component closure $\bar L$ of $L$,
shown on the left in Figure \ref{nonlocal}.
Clearly, if $L$ has a local knot, then so does $\bar L$.
Moreover, if $\bar L$ has a local knot, then this local knot must be one of the prime factors of its components.
But the only prime factor of the components of $\bar L$ is the trefoil knot.
Thus it suffices to show that $\bar L$ has no local knot that is the trefoil.

The Jones polynomial of $\bar L$ is $V_L(t)=t^{-7/2}-t^{-5/2}-t^{-1/2}-t^{1/2}-t^{5/2}+t^{7/2}$.
The Jones polynomial of the trefoil knot $T$ is $V_T(t)=t+t^3-t^4$.
If $T$ is a local knot in $\bar L$, then $V_L(t)$ must be divisible by $V_T(t)$ in $\Z[t^{\pm1/2}]$
(see \cite{Lic}*{p.\ 29}; compare \cite{Ro4}).
But it is easy to see that $V_L(t)=t^{-7/2}\big(t^4(t^3-t^2-1)-(t^3+t-1)\big)$ is not divisible by 
$V_T(t)=(-t)(t^3-t^2-1)$.
\end{example}

\begin{remark} Let us note that the non-locality of the trefoil in $\bar L$ is not detected by
the Conway polynomial and the Conway potential function.
In fact their values on a possibly twisted Whitehead double of a two-component link $(K,K')$ along $K$
depend only on $\lk(K,K')$ and on the knot $K'$ (see \cite{M24-3}*{Corollary \ref{part3:whitehead}}).
\end{remark}

For every (string) link $L$ there is a unique, up to ambient isotopy, totally split (string) link $K_L$ whose 
components are ambient isotopic to the respective components of $L$.

\begin{corollary} \label{rolfsen-lemma} If $L$ and $L'$ are 

(a) links, or 

(b) string links

\noindent
that are PL isotopic, and $K_L=K_{L'}$ up to an ambient isotopy, then $L$ and $L'$ are ambient isotopic.
\end{corollary}

Part (a) is due to D. Rolfsen \cite{Ro0}, who proved it by a different method, making use of the prime 
factorization of knots.%
\footnote{Here is a quote from his proof: ``[The given PL isotopy] has singularities, corresponding to 
the appearance and vanishing of knots. [...] The plan is to move all these bad points to a single level, 
say $t=1/2$, and then `cancel' them''.}
Rolfsen also observed that its assertion fails for links in $S^1\x S^2$ \cite{Ro1}*{Example 2}.

\begin{proof}
Theorem \ref{global} yields a decomposition $L=\Gamma(L)\#\Lambda(L)$ up to ambient isotopy, 
where $\Gamma(L)$ (the ``global part'') has no local knots and $\Lambda(L)$ (the ``local part'') is totally split.
It is easy to see that $\Gamma(L)$ is invariant under adding local knots to $L$.
Therefore it is invariant under PL isotopy.%
\footnote{It is not hard to see (cf.\ \cite{Ro2}*{Theorem 4.2}) that two links are PL isotopic if and only if
they are related by the equivalence relation generated by ambient isotopy and addition of local knots.}

Since $L$ and $L'$ are PL isotopic, $\Gamma(L)=\Gamma(L')$.
In particular, $K_{\Gamma(L)}=K_{\Gamma(L')}$.
But clearly $K_L=K_{\Gamma(L)}\#\Lambda(L)$, and similarly for $L'$.
Since $K_L=K_{L'}$, we obtain (using the uniqueness of the prime factorization of a knot) that $\Lambda(L)=\Lambda(L')$.
Since $L$ represents the ambient isotopy type of $\Gamma(L)\#\Lambda(L)$ and similarly for $L'$, we conclude that $L$ 
is ambient isotopic to $L'$.
\end{proof}

\begin{remark}
Corollary \ref{rolfsen-lemma}(a) is used in a subsequent paper \cite{M24-1} to prove the following:
{\it Finite type invariants separate string links if and only if finite type invariants separate knots and finite type invariants, 
well-defined up to PL isotopy, separate PL isotopy classes of string links.}
When string links are replaced with links, the corresponding assertion remains an open problem, although 
it is not hard show, using Corollary \ref{rolfsen-lemma}(b) and the Kontsevich integral, that its version for
rational finite type invariants holds \cite{M24-1}.
\end{remark}

\begin{remark}
The referee made the following remarkable comment on the last two sentences of the abstract of the present paper:

\smallskip
\begin{center}
\parbox{14cm}{\small
``It might be possible to derive the string link version [of Rolfsen's theorem] from [the statements of Rolfsen's and Hashizume's theorems] 
by considering not the closure but the double of a string link (that is, when the given string link is glued to its symmetric copy).''}
\end{center}
\smallskip

Presumably the double of a given string link $L$ is understood to lie in a solid torus that is standardly embedded in $S^3$.
If so, then the double of $L$ is the same as the closure of $L\#\rho L$, where $\rho L$ is the reflection of $L$. 
That is, $\rho L$ is the composition $[m]\x I\xr{\id\x r}[m]\x I\xr{L} D^2\x I\xr{\id\x r}D^2\x I$, where
$[m]=\{1,\dots,m\}$ and $r\:I\to I$ is an orientation reversing homeomorphism.

The doubling construction does seem to ``kill'' Example \ref{nonlocal-knot} in the sense that it seems very unlikely that there exists a string link 
with no local knots whose double has a local knot.
Nevertheless, I doubt that it is possible to derive the string link version of Rolfsen's theorem from the statements of Rolfsen's and Hashizume's theorems
by using doubling.
For consider the following example: There exists a string link whose closure is the Whitehead link (see Figure \ref{wh-string}, on the right) 
but whose double is an unlink (see Figure \ref{wh-string}, on the left).
\begin{figure}[h]
\includegraphics[width=0.7\linewidth]{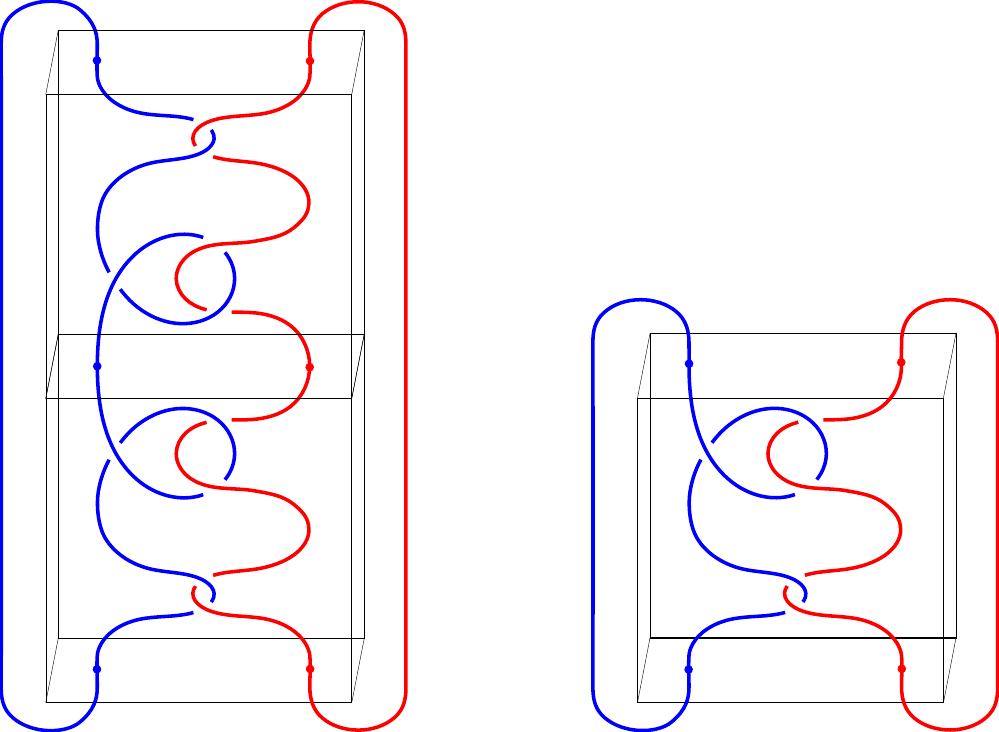}
\caption{A string link whose closure is the Whitehead link but whose double is an unlink.}
\label{wh-string}
\end{figure}
\end{remark}

In what follows we will often abuse notation by denoting links, string links and tangles by what really are their images.
These images are always considered to be oriented submanifolds of the oriented $3$-manifold and to have ordered components
(even if they appear to be treated as subsets).

\section{Split links}

A link $L$ is called {\it split} if there exists a PL $3$-ball $B\subset S^3$ 
such that $L$ is disjoint from $\partial B$ but meets both $B$ and $S^3\but B$.
A string link $L\:mI\to I^3$ is called {\it split} if there exists a properly embedded PL $2$-disk 
$D\subset I^3$ such that $L$ is disjoint from $D$ but meets both components of $I^3\but D$.

\begin{proposition} \label{partition}
(a) Given a link $L$, there exists a family of pairwise disjoint PL $3$-balls $B_i$ in $S^3$ such that $L$ lies in their
union and meets each $B_i$ in a non-split sublink.

(b) Given a string link $L$, there exists a family of pairwise disjoint properly embedded PL $2$-disks in $I^3$
such that $L$ is disjoint from their union $\Delta$ and meets each component of $I^3\but\Delta$ in a non-split sublink.
\end{proposition}

Part (a) is presumably well-known, but I could not find it in textbooks.
Part (b) will not be used in the remainder of the paper. 
But for completeness we prove both.

\begin{proof}[Proof. (a)]
Let $B_0$ be some PL ball containing $L$.
If $L$ is non-split, this completes the proof.
To proceed we need the following lemma.

\begin{lemma} \label{split0}
If $L$ is split and lies in a PL ball $B_0$, then there exist disjoint PL balls $B_+$ and $B_-$ in $B_0$ such that $L$ lies 
in their union and meets both of them.
\end{lemma}

\begin{proof} Since $L$ is split, exists a PL ball $B\subset S^3$ such that $L$ is disjoint from $\partial B$ but meets both $B$ and 
$S^3\but B$.
Let $B_1$ be a PL ball containing both $B_0$ and $B$ in its interior.
By the PL annulus theorem there exists a PL homeomorphism $B_1\to B_0$ keeping $L$ fixed.
It sends $B$ onto a PL ball $B_+$ lying in the interior of $B_0$.
Since $L$ is a $1$-manifold, there exists an arc $J$ in $\overline{B_0\but B_+}\but L$ joining $\partial B_+$ to $\partial B_0$.
Let $N$ be a regular neighborhood of $B_+\cup J$ in $B_0$, disjoint from $L$, and let $B_-=\overline{B_0\but N}$.
\end{proof}

Let us note that each of $L\cap B_+$ and $L\cap B_-$ has fewer components than $L$.
Hence by repeatedly applying Lemma \ref{split0} we eventually obtain a desired family of balls.
\end{proof}

\begin{proof}[(b)] The assertion follows similarly to (a) from the following lemma.

\begin{lemma} If $L$ is split and lies in one component of $I^3\but D_0$ for some properly embedded PL $2$-disk 
$D_0\subset I^3$, then there exists a properly embedded PL $2$-disk $D\subset I^3$, disjoint from $D_0$ and such that 
$L$ is disjoint from $D$ but meets both components of $I^3\but D$.
\end{lemma}

\begin{proof}
Since $L$ is split, there exists a properly embedded PL $2$-disk $D'\subset I^3$ such that $L$ is disjoint from $D'$ 
but meets both components of $I^3\but D'$.
On the other hand, let $B_0$ be the closure of the component of $I^3\but D_0$ containing $L$.
By Alexander's Schoenflies theorem the $2$-sphere $\partial B_0$ is unknotted in $S^3$, and hence $B_0$
is a PL $3$-ball.
Then there is a PL homeomorphism $h\:B_0\to I^3$ which is the identity outside a small neighborhood of $D_0$ in $B_0$.
In particular, this neighborhood can be chosen to be disjoint from $L$.
We may further assume that $h(D_0)$ is a small disk, which is disjoint from $D'$.
Then $D:=h^{-1}(D')$ is disjoint from $D_0$ and retains the properties of $D$.
\end{proof}
\end{proof}

\section{Prime factorization} \label{factorization}

By a {\it punctured $3$-sphere} we shall mean the $3$-manifold $P_n:=\overline{S^3\but(B_1\cup\dots\cup B_n)}$ for some $n\ge 0$,
where $B_1,\dots,B_n$ are pairwise disjoint PL balls.
A {\it beaded link}, or more specifically an {\it $n$-beaded link}, consists of an oriented $1$-manifold $\Theta$ and a tangle 
$L\:\Theta\to P_n$ that sends precisely two points of $\partial\Theta$ of the opposite signs into each $\partial B_k$.
Let us note that there exists a unique orientation-reversing identification 
$\partial\Theta=\partial (nI)$ such that $L$ sends each $\{k\}\x\partial I$ into $\partial B_k$.
The {\it closure} of a beaded link $L$ is the link in $S^3$ obtained by gluing $L$ with any embedding 
$\Xi\:nI\to\partial P_n$ that sends each $\{k\}\x I$ into $\partial B_k$.
It is easy to see that the choice of a particular embedding $\Xi$ is irrelevant in the sense that the closure is well-defined 
up to ambient isotopy.
We call a beaded link a {\it beaded (un)knot} if its closure is a knot (respectively, an unknot).
The {\it beaded components} of a beaded link $L$ are the beaded knots whose closures are the components of a closure of $L$.
The {\it genus} $g(L)$ of a beaded link $L$ is the genus of its closure $\bar L$, that is, the minimal genus of a {\it connected} Seifert surface
spanned by the link $\bar L$.

Given an $n$-beaded link $L\:\Theta\to P_n$ and a PL $2$-sphere $\Sigma$ in the interior of $P_n$ that meets $L(\Theta)$ 
transversely in two points, by Alexander's Schoenflies theorem it bounds two complementary $3$-balls in $S^3$ and hence 
partitions $L$ into an $(i+1)$-beaded link $L'$ and an $(n-i+1)$-beaded link $L''$ for some $i\le n$.
When $n=0$, the link $L$ can be easily reconstructed from the (oriented) $1$-beaded links $L'$ and $L''$, and in this case 
we write $L=L'\#L''$.

On the other hand, given a link $L\:mS^1\to S^3$ and an index $i\in\{1,\dots,m\}$, we can represent $L$ as the closure of 
a $1$-beaded link $L[i]\:(mS^1\but\{i\}\x S^1)\cup I\to B^3$, which is easily seen to be well-defined up to ambient isotopy.
Given another link $L'\:m'S^1\to S^3$ and another index $j\in\{1,\dots,m'\}$, we get the link
$L[i]\#L'[j]\:(m+m'-1)S^1\to S^3$, which is known as the Hashizume connected sum $L\#_{i,j} L'$ \cite{Has}.
It is easily seen to yield a well-defined operation on pairs $([L],i)$, where $[L]$ is the ambient isotopy type of the link $L$ in $S^3$, 
and $i$ is a selected component of $L$.
Moreover, it is easy to see that this operation is associative and that the unknot is its unit.

If $K$ and $K'$ are knots, then $K\#_{1,1}K'$ is their usual connected sum $K\#K'$.
It is well-known that connected sum of knots is commutative, and the same proof (see \cite{BZ}*{Figure 7.3} or 
\cite{Fox-qt}*{p.\ 140}) works for the Hashizume connected sum.
It is also well-known that $g(K\#K')=g(K)+g(K')$, and the 13-line proof of this fact in \cite{Fox-qt}*{p.\ 141} 
works verbatim to show that $g(L\#_{i,j}L')=g(L)+g(L')$ for any links $L$, $L'$ and any $i$, $j$.
This in turn implies that whenever a PL $2$-sphere $\Sigma$ in the interior of $P_n$ meets an $n$-beaded link $L$ transversely
in two points and hence partitions it into beaded links $L'$ and $L''$, we have $g(L)=g(L')+g(L'')$.
Using this, it is easy to prove the non-cancellation property:

\begin{lemma} \label{non-cancellation}
Let $L$ be an $n$-beaded link, and suppose that the interior of $P_n$ contains a PL $2$-sphere $\Sigma$ that meets 
$L$ transversely in two points and partitions it into beaded links $L_1$ and $L_2$ so that $L_2$ is not a beaded unknot.
Then $L_1$ has either fewer beaded components than $L$ or genus smaller than that of $L$.
In particular, $L$ is not a beaded unknot.
\end{lemma}

\begin{proof}
If $L_1$ has the same number of beaded components as $L$, then $L_2$ is a beaded knot.
But $L_2$ is not a beaded unknot, so $g(L_2)\ge 1$.
Then $g(L)=g(L_1)+g(L_2)$ implies $g(L_1)<g(L)$.
\end{proof}

A link $L$ (in $S^3$) is called {\it prime} if it is non-split and not an unknot, but for every representation of $L$ as $A\#_{i,j}B$, 
either $A$ or $B$ is the unknot.
$L$ is {\it composite} if it is non-split, not an unknot and not a prime link.
Thus a non-split link $L$ is composite if and only if $S^3$ contains a PL $2$-sphere that meets $L$ in two points and partitions it into 
two $1$-beaded links $L_1$, $L_2$ neither of which is a beaded unknot.
(The ``if'' assertion uses the case $n=0$ of Lemma \ref{non-cancellation}.)

\begin{theorem}[Hashizume \cite{Has}] \label{hashizume}
Every non-split link $L$ admits a unique decomposition into prime factors $L_0,\dots,L_r$ (well-defined up to ambient isotopy
and up to a permutation of the factors) with respect to the operation of Hashizume connected sum.

Moreover, $S^3$ contains $r$ pairwise disjoint $2$-spheres, each meeting $L$ transversely in two points, 
that partition $L$ into $r+1$ beaded links whose closures are $L_0,\dots,L_r$.
\end{theorem}

\begin{proof} If $L$ is prime, there is nothing to prove.
Else we get the desired collection of $r$ spheres by repeatedly applying the following lemma along with Lemma \ref{non-cancellation}. 

\begin{lemma} \label{separation}
Let $L$ be an $n$-beaded link whose closure is a composite link.
Then the interior of $P_n$ contains a PL $2$-sphere $\Sigma$ that meets $L$ transversely in two points and partitions it 
into beaded links $L_1$, $L_2$, neither of which is a beaded unknot.
\end{lemma}

\begin{proof} The case $n=0$ holds by the definition of a composite link.
Suppose that the lemma holds for $(n-1)$-beaded links.
By adjoining to $L$ an arc $J$ in $\partial B_n$ we obtain an $(n-1)$-beaded link $\bar L$ in $P_{n-1}=P_n\cup B_n$.
We may identify $B_n$ with the cone over its boundary, $c*\partial B_n$.
Then there is an ambient isotopy of $P_{n-1}$ keeping $L$ fixed, moving $J$ along $c*J$ and taking it onto $c*\partial J$.
Let $\hat L$ be the $(n-1)$-beaded link in $P_{n-1}$ obtained by adjoining $c*\partial J$ to $L$.
Now by the induction hypothesis the interior of $P_{n-1}$ contains a PL $2$-sphere $\Sigma'$ that meets $\hat L$ 
transversely in two points and partitions it into two beaded links neither of which is a beaded unknot.
We may assume that $\Sigma'$ is disjoint from the cone point $c$.
Let $h$ be a self-homeomorphism of the pair $(P_{n-1},\hat L)$ with support in a small neighborhood of $B_n$
that squeezes $B_n$ radially onto a small ball $\epsilon B_n$ about $c$, disjoint from $\Sigma'$. 
Then $\Sigma:=h^{-1}(\Sigma')$ lies in the interior of $P_n$, meets $\hat L$ transversely in two points and partitions it 
into two beaded links neither of which is a beaded unknot.
The same holds if we replace $\hat L$ with $\bar L$, and consequently also if we replace it with $L$.
\end{proof}

The uniqueness of the factorization now follows from the following lemma.

\begin{lemma}\label{punctured spheres}
Let $L$ be a non-split link and let $\Sigma$ be the union of a collection of $r$ pairwise disjoint PL $2$-spheres $\Sigma_i$ in $S^3$, 
each meeting $L$ transversely in two points, such that the $r+1$ beaded links cut out of $L$ by $\Sigma$ have prime closures.
Let $L/\Sigma$ denote the unordered collection of the ambient isotopy types of these closures.

(a) Let $\Sigma'$ be another union of the same kind.
Then there exists yet another union $\Sigma''$ of the same kind such that $\Sigma''\cap\Sigma=\emptyset$ and
$L/\Sigma''=L/\Sigma'$.

(b) $L/\Sigma''=L/\Sigma$.
\end{lemma}

\begin{proof}[Proof. (a)] We may assume that $\Sigma$ intersects $\Sigma'$ transversely in a closed $1$-manifold $M$.
The proof is by induction on the number of components of $M$.
If $M$ is empty, then we let $\Sigma''=\Sigma'$.
Suppose that $M$ is non-empty.
As long as we fix some orientations on $S^1\sqcup\dots\sqcup S^1$, on $S^3$ and on $\Sigma$, each component of $\Sigma$ meets $L$ in two points 
of the opposite signs, and $\Sigma\cap L=\Delta_+\cup\Delta_-$, where $\Delta_+$ consists 
of the positive intersections and $\Delta_-$ of the negative ones.
Now $M$ contains a component $C$ that is innermost in $\Sigma\but\Delta_+$ in the sense that $C$ bounds a disk $D$ 
in $\Sigma\but\Delta_+$ whose interior is disjoint from $M$.
Thus the interior of $D$ is also disjoint from $\Sigma'$.
Let $D_+$ and $D_-$ be the two disks bounded by $C$ in $\Sigma'$.
The $2$-spheres $D\cup D_+$ and $D\cup D_-$ are unknotted in $S^3$ by Alexander's Schoenflies theorem, and hence bound in $S^3$ 
balls $B_+$ and $B_-$ (respectively) such that $B_+\cap B_-=D$ and $B_+\cup B_-$ is a $3$-ball bounded by the $2$-sphere $D_+\cup D_-$.
Let $B$ be the other $3$-ball bounded in $S^3$ by the latter $2$-sphere, and let $H$ be a $2$-handle attached to $B$ along $D$.
We may assume that $H$ is thin enough to be either disjoint from $L$ (if $D$ is disjoint from $\Delta_-$) or to meet $L$
in an arc $J$ such that $(H,D,J)$ is PL homeomorphic to $\big([-1,1]^3,\,[-1,1]^2\x\{0\},\,\{0\}^2\x[-1,1]\big)$ (if $D$ meets $\Delta_-$).
Then $\partial(B\cup H)$ consists of two disjoint $2$-spheres $S_+$ and $S_-$, which lie in $B_+$ and $B_-$ respectively.

If $D$ is disjoint from $\Delta_-$, then $C$ has zero linking number with every component of $L$, and consequently both points of
intersection between $L$ and the $2$-sphere $D_+\cup D_-$ must lie in one of the disks $D_+$ and $D_-$, say in $D_+$.
Then the $2$-sphere $D\cup D_-$ is disjoint from $L$.
Since $L$ is non-split, at least one of the $3$-balls $B_-$ and $B\cup B_+$ bounded by this $2$-sphere must also 
be disjoint from $L$.
But $L$ definitely meets $D_+=B\cap B_+$, so it must be disjoint from $B_-$.
Then we define $\Sigma''$ as the modification of $\Sigma'$ obtained by replacing the $2$-sphere $D_+\cup D_-$ with 
the $2$-sphere $S_+$.
Then the intersection of $\Sigma''$ with $\Sigma$ is a subset of $M\but C$, and hence has fewer components than $M$.
Since $L$ is disjoint from $B_-$, it it easy to see that $L/\Sigma''=L/\Sigma'$.

If $D$ meets $\Delta_-$, then it meets $L$ in one point, and consequently both $D_+$ and $D_-$ must also meet 
the same component of $L$, each in one point (by considering the linking numbers).
The closure of the component of $S^3\but\Sigma'$ that contains the interior of $D$ is a punctured $3$-sphere $P_n$
(for some $n\ge 1$) such that $\partial P_n$ contains the $2$-sphere $D_+\cup D_-$.
Now $P_n$ intersects $L$ in an $n$-beaded link $\Lambda$, which is in turn partitioned by $D$ into 
an $(i+1)$-beaded link $\Lambda_+$ in the punctured $3$-sphere $P_n\cap B_+$, and an $(n-i)$-beaded link 
$\Lambda_-$ in the punctured $3$-sphere $P_n\cap B_-$, for some $i\le n$.
It is easy to see that the closure of $\Lambda$ is a Hashizume connected sum of the closures of $\Lambda_+$ and $\Lambda_-$.
Since the closure of $\Lambda$ is prime, some $\Lambda_\epsilon$ must be a beaded unknot, say, $\Lambda_-$.
Then we define $\Sigma''$ as the modification of $\Sigma'$ obtained by replacing the $2$-sphere $D_+\cup D_-$ 
with the $2$-sphere $S_+$. 
(It is easy to see that $S_+$ lies in $P_n\cap B_+$).
The intersection of $\Sigma''$ with $\Sigma$ is a subset of $M\but C$, and hence has fewer components than $M$.
Since $L$ meets $P_n\cap B_-$ in a beaded unknot, it is easy to see that $L/\Sigma''=L/\Sigma'$.
\end{proof}

\begin{proof}[(b)] Let $Q_0,\dots,Q_r$ be the closures of the components of $S^3\but\Sigma$ and let $R_0,\dots,R_s$ be those of $S^3\but\Sigma''$.
Since $\Sigma\cap\Sigma''=\emptyset$, each $P_{ij}:=Q_i\cap R_j$ is either empty or a punctured $3$-sphere (indeed, $S^3\but(Q_i\cap R_j)$ equals 
$(S^3\but Q_i)\cup(S^3\but R_j)$, which is a union of balls whose boundaries are pairwise disjoint).
Since each $R_j\cap R_k$ is either empty or a common boundary component, the same is true of each $P_{ij}\cap P_{ik}$; and similarly
for each $P_{ij}\cap P_{kj}$.
Let $L_i$ and $L'_j$ be the closures of the beaded links cut out of $L$ by $Q_i$ and $R_j$, respectively. 
Each $P_{ij}$ intersects $L$ either in $\emptyset$ or in a beaded link, whose closure will be denoted $\Lambda_{ij}$.
(The closure of $\emptyset$ is formally set to be $\emptyset$.)
Then each $L_i$ can be described as a certain Hashizume connected sum of those links $\Lambda_{ij}$, $j=1,\dots,s$, that are nonempty.
Since $L_i$ is prime, it follows by induction that there exists a $j_i$ such that $\Lambda_{ij}$ is either empty or the unknot
for all $j\ne j_i$, whereas $\Lambda_{ij_i}=L_i$ (up to ambient isotopy).
Similarly, for each $j$ there exists an $i_j$ such that $\Lambda_{ij}$ is either empty or the unknot for all $i\ne i_j$, whereas 
$\Lambda_{i_jj}=L'_j$.
Since $L_i$ is not an unknot, $i_{j_i}=i$, and since $L'_j$ is not an unknot, $j_{i_j}=j$.
Thus $i\mapsto j_i$ is a bijection such that $L_i=L'_{j_i}$.
\end{proof}
\end{proof}

\section{Local knots} \label{local knots}

\begin{proof}[Proof of Theorem \ref{global}. The case of links]
Given an $m$-component link $L$, let us consider its unique partition into non-split sublinks (Proposition \ref{partition}(a)) 
and the unique decomposition of each non-split sublink into prime factors (Theorem \ref{hashizume}).
By discarding those prime factors that are knots, we obtain an $m$-component link $\Gamma(L)$.
The discarded factors can also be put together to form a totally split $m$-component link $\Lambda(L)$.
Clearly, $\Gamma(L)$ has no local knots, and $L=\Gamma(L)\#\Lambda(L)$ (up to ambient isotopy).
\end{proof}

\begin{proof}[The case of string links]
It suffices to show that every string link (not necessarily non-split) admits a unique decomposition into a connected sum of 
(i) a string link that has no local knots and (ii) prime local knots (the latter being well-defined up to a permutation).
This can be done along the lines of the proof of Theorem \ref{hashizume}.

In more detail, we are given a string link $L\:mI\to I^3$.
The most important amendment in the proof is that all the punctured spheres will now be contained in the interior of $I^3$,
apart from the ``outer'' one, which lies in $I^3$ and contains $\partial I^3$; this outer punctured sphere will intersect $L$
in a tangle with no local knots, and the ``inner'' ones in beaded knots whose closures are prime.
If the genus of a string link is defined to be the genus of its closure, then the proof of the existence of 
a decomposition goes through without essential changes.
The amended proof of Lemma \ref{punctured spheres}(a) still has two cases: $D$ is disjoint from $\Delta_-$
and $D$ meets $\Delta_-$.
In the first case we use that, even though the given string link $L$ may be non-split, it contains no closed components.
In the second case we now have to consider two subcases: $P_n$ is an inner punctured sphere and $P_n$ is the outer punctured sphere.
The first subcase is treated like before.
In the second subcase one of the two balls $B_+$, $B_-$ lies in the interior of $I^3$.
If $B_\epsilon$ lies in the interior of $I^3$, then since $P_n$ intersects $L$ in a tangle with no local knots, 
$\Lambda_\epsilon$ must be a beaded unknot.
The remainder of the proof of the second subcase proceeds like before.
The amended proof of Lemma \ref{punctured spheres}(b) goes through with aid of the following observation:
if $Q_i$ is the outer punctured sphere and $R_j$ is an inner punctured sphere, then $\Lambda_{ij}$ must be either empty
or the unknot, using that $L_i$ is a tangle with no local knots.
\end{proof}

\section*{Acknowledgements}

I would like to thank A. Zastrow for sending me a copy of \cite{Ro0}. 
I'm also grateful to the referee for useful remarks.

\end{document}